\documentclass[psamsfonts, draft]{amsart}
\usepackage{amssymb,amsmath}
\usepackage{amsthm}
\usepackage{mathrsfs}
\usepackage{enumerate, indentfirst}
\usepackage[fleqn,tbtags]{mathtools}
\usepackage{courier}
\usepackage{hyperref}

\usepackage{xr}

\hypersetup{
    citecolor=black,%
    filecolor=black,%
    linkcolor=black,%
    pdffitwindow=true,%
    colorlinks=false,%
    unicode=true,
    }
\newtheoremstyle{mytheorem}%
{5pt}%
{3pt}%
{\itshape}%
{1pt}%
{\bf}%
{.}%
{.5em}%
{}%

\newtheoremstyle{myremark}%
{5pt}%
{3pt}%
{\upshape}%
{1pt}%
{\em}%
{.}%
{.5em}%
{}%

\newtheoremstyle{myexample}%
{5pt}%
{3pt}%
{\upshape}%
{1pt}%
{\bf}%
{.}%
{.5em}%
{}%

\theoremstyle{mytheorem}
\newtheorem{theorem}{Theorem}[section]

\newtheorem{lemma}[theorem]{Lemma}
\newtheorem{proposition}[theorem]{Proposition}
\newtheorem{corollary}[theorem]{Corollary}

\theoremstyle{myremark}

\theoremstyle{myexample}
\newtheorem{example}[theorem]{Example}
\numberwithin{equation}{section}
\makeatletter \renewenvironment{proof}[1][\proofname] {\par\pushQED{\qed}\normalfont\topsep6\p@\@plus6\p@\relax\trivlist\item[\hskip\labelsep\itshape #1\@addpunct{.}]\ignorespaces}{\popQED\endtrivlist\@endpefalse} \makeatother

\renewcommand{\phi}{\varphi}
\renewcommand{\theta}{\vartheta}
\renewcommand{\epsilon}{\varepsilon}

\DeclareMathOperator{\supp}{supp}

\newcommand{\lmu}{\mathscr{L}^2(\mu)}
\newcommand{\lnu}{\mathscr{L}^2(\nu)}

\newcommand{\Abs}[1]{\big\lvert#1\big\rvert}
\newcommand{\alg}{\mathscr{A}}
\newcommand{\ring}{\mathscr{R}}

\DeclareMathOperator{\sform}{\mathfrak{s}}
\DeclareMathOperator{\tform}{\mathfrak{t}}
\DeclareMathOperator{\mform}{\mathfrak{m}}
\DeclareMathOperator{\nform}{\mathfrak{n}}
\DeclareMathOperator{\wform}{\mathfrak{w}}
\DeclareMathOperator{\vform}{\mathfrak{v}}
\DeclareMathOperator{\aform}{\mathfrak{a}}
\DeclareMathOperator{\bform}{\mathfrak{b}}

\newcommand{\abs}[1]{\lvert#1\rvert}
\newcommand{\dupN}{\mathbb{N}}

\newcommand{\seq}[1]{(#1_{n})_{n\in\dupN}}

\newcommand{\dupC}{\mathbb{C}}

\newcommand{\dom}{\operatorname{dom}}

\newcommand{\D}{\mathscr{D}}

\newcommand{\hil}{\mathfrak{H}}
\newcommand{\hils}{\hil_{\mathfrak{s}}}
\newcommand{\hilt}{\hil_{\mathfrak{t}}}
\newcommand{\hilm}{\mathfrak{H}_{_\mu}}

\newcommand{\hila}{\mathfrak{H}_{\mathfrak{a}}}
\newcommand{\hilb}{\mathfrak{H}_{\mathfrak{b}}}
\newcommand{\hilw}{\mathfrak{H}_{w}}
\newcommand{\hilv}{\mathfrak{H}_{v}}

\DeclarePairedDelimiterX\sip[2]{(}{)}{#1\,\delimsize\vert\,#2}
\DeclarePairedDelimiterX\siptilde[2]{(}{)_{\!_{\widetilde{A}}}}{#1\,\delimsize\vert\,#2}
\DeclarePairedDelimiterX\sipn[2]{(}{)_{\nu}}{#1\,\delimsize\vert\,#2}
\DeclarePairedDelimiterX\sipm[2]{(}{)_{\mu}}{#1\,\delimsize\vert\,#2}
\DeclarePairedDelimiterX\sips[2]{(}{)_{\sform}}{#1\,\delimsize\vert\,#2}
\DeclarePairedDelimiterX\sipt[2]{(}{)_{\tform}}{#1\,\delimsize\vert\,#2}
\DeclarePairedDelimiterX\set[2]{\{}{\}}{#1\,\delimsize\vert\,#2}
\DeclarePairedDelimiterX\dual[2]{\langle}{\rangle}{#1,#2}
\DeclarePairedDelimiterX\sipa[2]{(}{)_{\aform}}{#1\,\delimsize\vert\,#2}
\DeclarePairedDelimiterX\sipb[2]{(}{)_{\bform}}{#1\,\delimsize\vert\,#2}
\DeclarePairedDelimiterX\sipv[2]{(}{)_{v}}{#1\,\delimsize\vert\,#2}
\DeclarePairedDelimiterX\sipw[2]{(}{)_{w}}{#1\,\delimsize\vert\,#2}

\newcommand{\cha}[1]{\chi_{_{#1}}}

\newcommand{\limn}{\lim\limits_{n\rightarrow\infty}}

\begin{document}
\title[Radon--Nikodym theorems]{Radon--Nikodym theorems for nonnegative forms, measures and representable functionals}

\author[Zs. Tarcsay]{Zsigmond Tarcsay}
\address{Zs. Tarcsay, Department of Applied Analysis, E\"otv\"os L. University, P\'azm\'any P\'eter s\'et\'any 1/c., Budapest H-1117, Hungary; }
\email{tarcsay@cs.elte.hu}


\begin{abstract}
The aim of this note is to establish two Radon--Nikodym type theorems for nonnegative Hermitian forms defined on a real or complex vector space. We also apply these results to prove the known Radon--Nikodym theorems of the theory of representable positive functionals, $\sigma$-additive and finitely additive measures.
\end{abstract}

\maketitle

\section{Introduction}

 The notion of absolute continuity appears in several fields of the analysis. For example, in measure theory, a measure $\nu$ is said to be absolutely continuous with respect to another measure $\mu$ if $\mu(E)=0$ implies $\nu(E)=0$ for each measurable set $E$. Denoting by $\D$ the vector space of measurable step functions, one easily verifies that absolute continuity means that the canonical embedding map of $\D\subseteq\lmu$ into $\lnu$ is a well defined linear operator. It is less evident that $J$ is automatically closable in that case:
  \begin{equation*}
        \int \abs{\phi_n}^2~d\mu\to0\quad \textrm{and}\quad \int \abs{\phi_n-\phi_m}^2~d\nu\to0\quad\textrm{imply} \quad\int \abs{\phi_n}^2~d\nu\to0.
    \end{equation*}
For the proof see \cite[Lemma 5.1]{Hassi2009}. If we associate two nonnegative Hermitian forms on $\D$ with the measures under consideration, namely by setting
\begin{equation*}
    \mform(\phi,\psi)=\int \phi\overline{\psi}~d\mu, \qquad \nform(\phi,\psi)=\int \phi\overline{\psi}~d\nu,
\end{equation*}
then we see  that $\nform$ is $\mform$-absolutely continuous (see \cite{Hassi2009} or \cite{Simon}):
\begin{equation*}
    \mform(\phi_n,\phi_n)\to0\quad \textrm{and}\quad\nform(\phi_n-\phi_m,\phi_n-\phi_m)\to0\quad\textrm{imply} \quad\mform(\phi_n,\phi_n)\to0.
\end{equation*}

By considering (nonnegative valued) finitely additive set functions, the situation becomes somewhat more complicated: $\beta$ is called $\alpha$-absolutely continuous if for any $\epsilon>0$ there is $\delta>0$ such that $\beta(E)<\epsilon$ whenever $\alpha(E)<\epsilon$.
Just like in the $\sigma$-additive case, we can associate two forms, $\aform$ and $\bform$, with $\alpha$ and $\beta$, respectively:
  \begin{equation*}
    \aform(\phi,\psi)=\int \phi\overline{\psi}~d\alpha, \qquad \bform(\phi,\psi)=\int \phi\overline{\psi}~d\beta.
\end{equation*}
The symbol $\D$ stands again for the space of measurable step functions. It turns out that the $\alpha$-absolute continuity of $\beta$, under certain natural conditions, is equivalent to the $\aform$-absolute continuity of $\bform$, see \cite[Theorem 3.2]{sebestytarcsaytitkos} or \cite[Theorem 3.4]{titkos}. The proof of this statement in the cited papers uses a general Lebesgue decomposition theorem \cite[Theorem 2.11]{Hassi2009} concerning forms (cf. also \cite[Theorem 2.3]{sebestytarcsaytitkos}). For the sake of the reader we include an independent and self-contained proof here (see Lemma \ref{L:additivelemma} below).

The notion of absolute continuity occurs also in the theory of positive linear functionals on a $^*$-algebra: $w$ is called $v$-absolutely continuous if
\begin{equation*}
    v(a_n^*a_n)\to0\quad\textrm{and}\quad w((a_n^*-a^*_m)(a_n-a_m))\to0 \quad\textrm{imply} \quad w(a_n^*a_n)\to0.
\end{equation*}
We can, of course, associate two forms $\vform$ and $\wform$ with $v$ and $w$, respectively, in a quite natural way:
\begin{equation*}
    \vform(a,b)=v(b^*a),\qquad \wform(a,b)=w(b^*a).
\end{equation*}
The equivalence of the corresponding absolute continuity concepts is obvious in this case.

We conclude therefore that there is a very tight connection between the notions of absolute continuity in the above settings. Nevertheless, we can find a further deep relation: there appear also appropriate results regarding the representability of the absolutely continuous object via the dominating one. Such results used to be called  Radon--Nikodym type theorems. In the present paper we offer two Radon--Nikodym type theorems of different type for nonnegative Hermitian forms: the first one is motivated by the Radon--Nikodym--Darst theorem \cite{DarstGreen} for finitely additive set functions (cf. also  Fefferman \cite{Fefferman}), and the second one was suggested by a theorem due to Gudder \cite[Theorem 1]{Gudder} concerning positive linear functionals on a $^*$-algebra.

The remaining parts of this note contain several applications of these  results: In Section 3 we discuss the case of (finitely) additive set functions and prove that the Radon--Nikodym--Darst theorem \cite{DarstGreen} is a straightforward consequence of our main result Theorem \ref{T:maintheorem}. As a novelty we consider set functions defined on a \emph{ring of sets} in contrast to \cite{DarstGreen} or \cite{Fefferman} where set functions over \emph{algebras} were examined. Section 4 is devoted to the measure theory : we give an operator theoretic proof of the classical Radon--Nikodym theorem. Furthermore, we characterize the case when the Radon--Nikodym derivative belongs to $\mathscr{L}^2$. Finally, in Section 5 we provide two Radon--Nikodym type theorems for positive functionals on a $^*$-algebra. In particular, we extend a corresponding result due to Gudder \cite{Gudder} to the case of representable functionals  defined on a not necessarily unital $^*$-algebra and  we characterize the absolute continuity among pure functionals on a $C^*$-algebra.

\section{Radon--Nikodym theorems for nonnegative Hermitian forms}
 Let $\D$ be a vector space over the real or complex field. Given a nonnegative Hermitian form $\sform$ on $\D$, we always associate a Hilbert space $\hils$ with $\sform$ by the standard method. That is to say, by setting $\mathscr{N}_{\sform}=\set{x\in\D}{s(x,x)=0}$,  we define the Hilbert space $\hils$ to be the completion of the pre-Hilbert space $\D/\mathscr{N}_{\sform}$ endowed with the inner product
 \begin{equation*}
    \sips{x+\mathscr{N}_{\sform}}{y+\mathscr{N}_{\sform}}=\sform(x,y),\qquad x,y\in\D.
 \end{equation*}
 For simplicity of notation, we continue to write $x$ for $x+\mathscr{N}_{\sform}$, in the hope that we do not cause any confusion.

 If another nonnegative Hermitian form $\tform$ on $\D$ is given, then we say that $\sform$ is absolutely continuous (in other words, closable, cf. \cite{Hassi2009,Simon}) with respect to $\tform$ if
 \begin{equation}\label{E:closability}
    \tform(x_n,x_n)\to0\qquad\textrm{and}\qquad \sform(x_n-x_m,x_n-x_m)\to0
 \end{equation}
 imply $\sform(x_n,x_n)\to0$ for all sequence $\seq{x}$ of $\D$. Following the terminology of Gudder \cite{Gudder}, a sequence $\seq{x}$ satisfying \eqref{E:closability} will be called a $(\tform,\sform)$-sequence.

 Our main result in this section is the following Radon--Nikodym type theorem:
 \begin{theorem}\label{T:maintheorem}
    Let $\sform$ and $\tform$ be nonnegative Hermitian forms on the real or complex vector space $\D$ such that $\sform$ is absolutely continuous with respect to $\tform$. Then for each $f\in\hils$ there exists a sequence $\seq{x}$ of $\D$ such that
    \begin{equation}\label{E:pseudo}
        \sips{x}{f}=\limn \sipt{x}{x_n}\qquad \textrm{for all $x\in\D$}.
    \end{equation}
    Moreover, the convergence is uniform on the set $\set{x\in\D}{\sform(x,x)+\tform(x,x)\leq1}$.
 \end{theorem}
 \begin{proof}
    Absolute continuity of $\sform$ with respect to $\tform$ means precisely that the canonical embedding operator $J$ of $\D\subseteq\hilt$ into $\hils$, defined by
 \begin{equation}\label{E:J}
    Jx=x,\qquad x\in\D,
 \end{equation}
  is  closable. Hence the domain $\dom J^*$ of its adjoint is dense in $\hils$. Consider a sequence $\seq{g}$ of $\dom J^*$ such that $g_n\to f$ in $\hils$, and for any integer $n$ fix $x_n\in\D$ such that $\|J^*g_n-x_n\|_{\tform}<1/n$. Then for   $x\in\D$, $\sform(x,x)+\tform(x,x)\leq1$  we infer that
  \begin{align*}
    \abs{\sips{x}{f} - \sipt{x}{x_n}}&\leq \abs{\sips{x}{f} -\sips{Jx}{g_n}} + \abs{\sips{x}{J^*g_n}- \sipt{x}{x_n}}\\
    &\leq \|f-g_n\|_{\sform}+\|J^*g_n-x_n\|_{\tform}\to 0,
  \end{align*}
  as it is claimed.
 \end{proof}
 Hereinafter we shall call the form $\sform$-\emph{pseudo-absolutely continuous} with respect to $\tform$ if for any $f\in\hils$ there exists $\seq{x}$ of $\D$ satisfying \eqref{E:pseudo}. We can therefore reformulate Theorem \ref{T:maintheorem} as follows: absolute continuity implies pseudo-absolute continuity. It is a natural question  whether this statement can be reversed. As Example \ref{Ex:counter} below demonstrates, the answer is negative in general.

 Before giving a counterexample we briefly recall the notion of singularity: $\sform$ and $\tform$ are called singular with respect to each other if for each nonnegative Hermitian form $\wform$ the properties $\wform\leq\sform$ and $\wform\leq\sform$ imply $\wform=0$. By a general Lebesgue decomposition theorem due to Hassi, Sebesty\'en and de Snoo \cite{Hassi2009} (see also \cite{sebestytarcsaytitkos}) each nonnegative Hermitian form $\sform$ can be decomposed into a sum $\sform=\sform_a+\sform_s$ such that $\sform_a$ is absolutely continuous with respect to $\tform$ and that $\sform_s$ and $\tform$ are mutually singular. In particular, singularity and absolute continuity are "dual" concepts in the sense that if $\sform$ and $\tform$ are singular with respect to each other, and $\sform$ is  $\tform$-absolutely continuous then $\sform=0$.

 Below we present  an example demonstrating that a nonzero form can be simultaneously singular and pseudo-absolutely continuous with respect to another form:
 \begin{example}\label{Ex:counter}
    Let $\D$ stand for the algebra of all continuous $\dupC$-valued functions defined on $I=[-1,1]$. For fixed $\phi,\psi\in\D$ we define the forms $\sform$ and $\tform$ by
    \begin{equation*}
        \sform(\phi,\psi)=\phi(0)\overline{\psi(0)},\qquad \tform(\phi,\psi)=\int_I \phi\overline{\psi}~d\lambda.
    \end{equation*}
   Consider a sequence $\seq{\phi}$ of $\D$ satisfying $\phi_n\geq0, \supp \phi_n\subseteq[-1/n,1/n]$ and $\displaystyle \int_I \phi_n~d\lambda=1$. It is clear that
  \begin{align*}
    \sips{\phi}{1}=\phi(0)=\limn\int_I \phi_n~d\lambda=\limn \sipt{\phi}{\phi_n},
  \end{align*}
  for all $\phi\in\D$. Since $\hils$ is one-dimensional, one concludes that $\sform$ is pseudo-absolutely continuous with respect to $\tform$. Nevertheless, one easily verifies that $\sform$ and $\tform$ are singular.
 \end{example}
In the next theorem we present another Radon--Nikodym type theorem, suggested by \cite[Theorem 1]{Gudder} (cf. also \cite{sebestytitkosRN}),  which at the same time offers a characterization of the absolute continuity:
\begin{theorem}\label{T:RNforms}
    Let $\sform, \tform$ be nonnegative Hermitian forms on a complex vector space $\D$. The following statements are equivalent:
    \begin{enumerate}[\upshape (i)]
      \item $\sform$ is absolutely continuous with respect to $\tform$;
      \item There is a positive selfadjoint operator $S$ in $\hilt$ such that $\D\subseteq\dom S^{1/2}$ and that
      \begin{equation}\label{W1/2}
        \sipt{S^{1/2}x}{S^{1/2}y}=\sform(x,y),\qquad \textrm{for all $x,y\in\D$.}
      \end{equation}
    \end{enumerate}
\end{theorem}
\begin{proof}
    We are going to prove first that (ii) implies (i). Consider a $(\tform,\sform)$-sequence $\seq{x}$ of $\D$. In the language of Hilbert space operators that means that
    \begin{equation*}
        \sipt{x_n}{x_n}\to0\qquad \textrm{and}\qquad \sipt{S^{1/2}(x_n-x_m)}{S^{1/2}(x_n-x_m)}\to0,
    \end{equation*}
    by letting $n,m\to\infty$. Then, by the closability of $S^{1/2}$ we infer that
    \begin{equation*}
        \sform(x_n,x_n)=\sipt{S^{1/2}x_n}{S^{1/2}x_n}\to0.
    \end{equation*}
    Hence $\sform$ is absolutely continuous with respect to $\tform$. Conversely, by assuming $\sform$ to be $\tform$-absolutely continuous, we conclude that the canonical embedding operator $J$ of $\hilt$ into $\hils$ \eqref{E:J} is closable. Hence, by a celebrated theorem of Neumann, $S:=J^*J^{**}$ is a positive selfadjoint operator in $\hilt$ such that $\D=\dom J\subseteq\dom J^{**}=\dom S^{1/2}$ and that
    \begin{equation*}
        \sipt{S^{1/2}x}{S^{1/2}y}=\sips{J^{**}x}{J^{**}y}=\sips{Jx}{Jy}=\sform(x,y),
    \end{equation*}
    for all $x,y\in\D$.
\end{proof}
\section{The Radon--Nikodym--Darst theorem of additive set functions}

Throughout this section $T$ is a nonempty set and $\ring$ is a ring of subsets of $T$. Let $\beta$ be a (real or complex valued) additive set function on $\ring$. We assume in the sequel that $\beta$ is  bounded, that is to say,
\begin{equation*}
    M:=\sup_{E\in\ring} \abs{\beta(E)}<\infty.
\end{equation*}
Then  the total variation $\abs{\beta}$ of $\beta$ exists, and thus we can naturally associate a nonnegative Hermitian form $\mathfrak{b}$ with $\beta$ on the vector space $\D$ of the (real or complex valued, respectively) $\ring$-simple functions by letting
\begin{equation}\label{aform}
    \bform(\phi,\psi)=\int_T \phi\overline{\psi}~d\abs{\beta}, \qquad \phi, \psi\in\D.
\end{equation}
Furthermore, we associate a Hilbert space $(\hilb,\sipb{\cdot}{\cdot})$ with $\bform$ just as in the previous section. By the boundedness of $\beta$ one easily verifies that  the correspondence
\begin{equation*}
    \phi\mapsto \int_T \phi~ d\beta
\end{equation*}
defines a continuous linear functional on the dense linear manifold $\D$ of $\hilb$, namely, by the norm bound $\sqrt{M}$. Hence the Riesz representation theorem yields a unique representing vector $\widehat{\beta}\in\hilb$ such that
\begin{equation}\label{betahat}
    \int_T \phi~ d\beta=\sipb{\phi}{\widehat{\beta}},\qquad \phi\in\D.
\end{equation}

Let another (not necessarily bounded) nonnegative additive set function $\alpha$ on $\ring$ be given. We say that $\beta$ is absolutely continuous with respect to $\alpha$ if for any $\epsilon>0$ there exists $\delta>0$ such that $\alpha(E)<\delta$ implies $\abs{\beta}(E)<\epsilon$ for all $E\in\ring$, see \cite{DarstGreen, Fefferman}.
Let us denote by $\aform$ and $\hila$ the corresponding associated Hermitian form and Hilbert space, respectively. It makes sense then to speak about absolute continuity of the form $\bform$ with respect to the form $\aform$, and a natural question there arises: Is there any relation between the notions of absolute continuity of additive set functions and Hermitian forms? The answer is given in the following lemma:
\begin{lemma}\label{L:additivelemma} Let $\alpha$ and $\beta$ be nonnegative additive set functions on $\ring$ where $\beta$ is bounded. If $\beta$ is absolutely continuous with respect to $\alpha$ then $\bform$ is absolutely continuous with respect to $\aform$. If $\alpha$ is bounded too then the converse of the statement is also valid.
\end{lemma}
\begin{proof} A proof of this statement is found in \cite[Theorem 3.2]{sebestytarcsaytitkos} in the particular case when both set functions are bounded. We treat here therefore only the first assertion, by giving a new and independent proof that does not make use of the nontrivial apparat of the Lebesgue decomposition theory of forms \cite{Hassi2009}.

Let $\beta$ be $\alpha$-absolutely continuous. Then there exists a sequence $\seq{\beta}$ of nonnegative additive set functions such that $\beta_n\leq\beta$, $\beta_n(E)\to\beta(E)$ for all $E\in\ring$ and that $\beta_n\leq c_n\alpha$ for some sequence $\seq{c}$ of nonnegative numbers (see e.g. \cite[Lemma 3.1]{sebestytarcsaytitkos}). By the Riesz representation theorem there is $\widehat{\beta_n}\in\hilb$ such that
\begin{equation*}
    \int_T \phi~d\beta_n=\sipb{\phi}{\widehat{\beta_n}},\qquad \phi\in\D.
\end{equation*}
Similarly, for $\psi\in \D$ there exists $\psi.\widehat{\beta_n}\in\hilb$ such that
\begin{equation*}
    \int_T \phi\overline{\psi}~d\beta_n=\sipb{\phi}{\psi.\widehat{\beta_n}},\qquad \phi\in\D.
\end{equation*}
It is seen readily that $\|\psi.\widehat{\beta_n}\|_{\bform}\leq\|\psi.\widehat{\beta}\|_{\bform}$.
The embedding operator $J$ \eqref{E:J} is well defined by the absolute continuity, and we conclude by the following line of inequalities
\begin{align*}
    \abs{\sipb{\phi}{\psi.\widehat{\beta_n}}}\leq \int_T \abs{\phi\psi}~d\beta_n\leq c_n\|\psi\|_{\aform}\|\phi\|_{\aform},\qquad \phi\in\D,
\end{align*}
that $\psi.\widehat{\beta_n}\in\dom J^*$ for all $\psi\in\D$. Fix $\psi\in\D$. Since $\sipb{\phi}{\psi.\widehat{\beta_n}}\to \sipb{\phi}{\psi.\widehat{\beta}}$ for $\phi\in\D$ and $\|\psi.\widehat{\beta_n}\|_{\bform}\leq\|\psi.\widehat{\beta}\|_{\bform}$ we infer that $\sipb{f}{\psi.\widehat{\beta_n}}\to \sipb{f}{\psi.\widehat{\beta}}$ for all $f\in\hilb$. Consider now $g\in\{\dom J^*\}^{\perp}$. Then
\begin{equation*}
    \sipb{g}{\psi.\widehat{\beta}}=\limn\sipb{g}{\psi.\widehat{\beta_n}}=0, \qquad \psi\in\D.
\end{equation*}
Since $\D$ is dense in $\hilb$ by definition, we may choose $\seq{\phi}$ of $\D$ such that $\|g-\phi_n\|_{\bform}\to0$. Then for any $\psi\in\D$ 
\begin{equation*}
    0=\sipb{g}{\psi.\widehat{\beta}}=\limn \sipb{\phi_n}{\psi.\widehat{\beta}}=\limn\int_T \phi_n\overline{\psi}~d\beta=\limn\sipb{\phi_n}{\psi}=\sipb{g}{\psi},
\end{equation*}
whence $g=0$. That means that $\dom J^*$ is dense in $\hilb$ and therefore that $J$ is closable. In other words, $\bform$ is $\aform$-absolutely continuous.
\end{proof}

We are now in position to give an extension of the Radon--Nikodym--Darst theorem \cite[Theorem 1]{DarstGreen}:
\begin{theorem}
    Let $\alpha, \beta$ be real or complex valued additive set functions on a ring of sets $\ring$ such that $\alpha$ is nonnegative and $\beta$ is bounded. If $\beta$ is absolutely continuous with respect to $\alpha$ then there is a sequence $\seq{\phi}$ of $\ring$-simple functions such that
     \begin{equation}\label{E:RNadditive}
     \displaystyle \beta(E)=\limn \int_E \phi_n~d\alpha
     \end{equation}
    If $\alpha$ is bounded too, then by letting $\displaystyle \beta_n(E):=\int_T \phi_n~d\alpha$ we have at the same time
    \begin{equation}\label{E:supadditive}
        \sup\limits_{E\in\ring}\abs{\beta-\beta_n}(E)\to 0.
    \end{equation}
\end{theorem}
\begin{proof}
    By Lemma \ref{L:additivelemma} we conclude that $\bform$ is $\aform$-absolutely continuous, and hence also $\aform$-pseudo-absolutely continuous in the view of Theorem \ref{T:maintheorem}. Consequently, there exists a $\seq{\psi}$ of $\D$ such that
    \begin{align*}
        \beta(E)=\sipb{\cha{E}}{\widehat{\beta}} =\limn \sipa{\cha{E}}{\psi_n}=\limn  \int_E \overline{\psi_n}~d\alpha,
    \end{align*}
    which yields \eqref{E:RNadditive} by setting $\phi_n:=\overline{\psi_n}$. Suppose in addition $\alpha$ to be bounded and fix a constant $M$ with $\sup\limits_{E\in\ring}(\alpha(E)+\abs{\beta}(E))\leq M$. Then
    \begin{gather*}
        \sup_{E\in\ring}\abs{\beta(E)-\beta_n(E)}=\sup_{E\in\ring}\abs{\sipb{\cha{E}}{\widehat{\beta}}-\sipa{\cha{E}}{\psi_n}}\\
        \leq\sup_{\set{\phi\in\D}{\aform(\phi,\phi)+\bform(\phi,\phi)\leq M}}\abs{\sipb{\phi}{\widehat{\beta}}-\sipa{\phi}{\psi_n}}\to0,
    \end{gather*}
    according to Theorem \ref{T:maintheorem} again. That yields \eqref{E:supadditive}, indeed. 
\end{proof}
\section{The classical Radon--Nikodym theorem of measures}

In this section we present the classical Radon--Nikodym theorem of measures in the following setting. Let $\ring$ be a $\sigma$-algebra of some subsets of $T$ and let $\mu, \nu$ be finite measures on $\ring$. Similarly to the case of additive set functions, we may associate the nonnegative Hermitian forms $\mform$ and $\nform$ with $\mu$ and $\nu$, respectively,  on the set $\D$ of simple functions. The corresponding Hilbert spaces $\hilm$ and $\hilm$ realize then as the well known function spaces $\lmu$ and $\lnu$, respectively, thanks to the celebrated Riesz--Fischer theorem.

Recall that $\nu$ is called absolutely continuous with respect to $\mu$ if $\mu(E)=0$ implies $\nu(E)=0$ for any measurable set $E$. Our purpose in this section is to prove the classical Radon--Nikodym theorem on the representability of $\nu$ due to $\mu$. Our method of proving is based on the analysis of the canonical embedding operator $J$ of $\lmu$ to $\lnu$.

First of all observe that $\nform$ is $\mform$-absolutely continuous due to Lemma \ref{L:additivelemma}. (For a more elementary proof, employing only classical arguments of the measure theory we refer to \cite[Lemma 5.1]{Hassi2009}.) This means in other words that $J$ is a closable operator from $\lmu$ into $\lnu$, and  therefore the domain of its adjoint $J^*$ is dense in $\lnu$. In the following lemma we investigate $J^*$ in detail:
\begin{lemma}\label{L:measureLemma2}
    Suppose that $\nu$ is absolutely continuous with respect to $\mu$. Then we have the following on $J^*$:
    \begin{enumerate}[\upshape (a)]
      \item $f\in\dom J^*$ implies $\abs{f}\in\dom J^*$.
      \item $J^*$ is a positive operator in the following sense: $f\in\dom J^*, f\geq0$ $\nu$-a.e. implies $J^*f\geq0$ $\mu$-a.e.
      \item If $f,g\geq0$ $\nu$-a.e., $f\in\dom J^*$, $g\in\lnu$ then $f\wedge g\in\dom J^*$.
    \end{enumerate}
\end{lemma}
\begin{proof}
    It is seen readily that the following inequality
    \begin{equation*}
        \abs{\sipn{\abs{f}}{\phi}}\leq\sup_{\psi\in\D,\abs{\psi}\leq\abs{\phi}}\abs{\sipn{f}{\psi}}
    \end{equation*}
     holds for any $f\in\lnu$ and $\phi\in\D$.
    Hence for $f\in\dom J^*$ we have
    \begin{equation*}
        \abs{\sipn{\abs{f}}{\phi}}\leq \|J^*f\|_{\mu}\|\phi\|_{\mu},\qquad \phi\in\D,
    \end{equation*}
    which yields $\abs{f}\in\dom J^*$. This proves (a). Let $f\in\dom J^*$, $f\geq0$ $\nu$-a.e. Then for each $\phi\in\D$, $\phi\geq0$ we have
    \begin{equation*}
        \sipm{J^*f}{\phi}=\sipn{f}{\phi}\geq0.
    \end{equation*}
    Hence $J^*f\geq0$ $\mu$-a.e. which yields (b). Finally, if  $f,g\geq0$ $\nu$-a.e., $f\in\dom J^*$, $g\in\lnu$ then
    \begin{align*}
        \abs{\sipn{f\wedge g}{\phi}}\leq \sipn{f}{\abs{\phi}}\leq \|J^*f\|_{\mu}\|\phi\|_{\mu}
    \end{align*}
     for all $\phi\in\D$. Hence $f\wedge g\in\dom J^*$ which proves (c).
\end{proof}
We are now in position to prove the main result of the section. Our treatment below is probably not the easiest way to get the Radon--Nikodym derivative. However, it may be interesting for an operator theorist.
\begin{theorem}
    Assume that $\nu$ is absolutely continuous with respect to $\mu$. Then there exists a $\mu$-integrable function $f$ such that
    \begin{equation*}
        \nu(E)=\int_E f~d\mu,\qquad E\in\ring.
    \end{equation*}
\end{theorem}
\begin{proof}
   By the closability of $J$, there exists $\seq{g}$ of $\dom J^*$ such that $g_n\to 1$ in $\lnu$ and simultaneously, $g_n\to 1$ $\nu$-a.e. Set
   \begin{equation*}
    f_n:= 1\wedge\left(\bigvee_{k=1}^{n} \abs{g_n}\right).
   \end{equation*}
   Then $f_n\leq f_{n+1}$ and that $f_n\to 1$ $\nu$-a.e.,  whence $f_n\to1$ in $\lnu$, by the B. Levi theorem.  At the same time, $f_n\in\dom J^*$, thanks to Lemma \ref{L:measureLemma2} (a) and (c). The positivity of $J^*$ (Lemma \ref{L:measureLemma2} (b)) yields then
   \begin{align*}
    \int_T\abs{J^*f_n-J^*f_m}~d\mu=\sipm{J^*(f_n-f_m)}{1}=\sipn{f_n-f_m}{1}, \qquad n\geq m.
\end{align*}
   By letting $n,m\to\infty$ we infer that $\seq{J^*f}$ converges to a ($\mu$-a.e. nonnegative) function $f\in\mathscr{L}^1(\mu)$, due to the Riesz--Fischer theorem. Then $f$ satisfies \begin{align*}
    \int_E f~d\mu=\limn\int_E J^*f_n~d\mu=\limn\sipm{J^*f_n}{\cha{E}}=\limn\sipn{f_n}{\cha{E}}=\nu(E)
\end{align*}
for all $E\in\ring$. The proof is therefore complete.
\end{proof}

By Lemma  \ref{L:measureLemma2} one concludes that $\dom J^*$ is a linear function lattice which possesses the Stone property: $1\wedge f\in\dom J^*$ for $f\in\dom J^*$. Nevertheless, $1$ does not belongs to $\dom J^*$ in general. As it turns out from Proposition \ref{P:RNderivative} below, $1\in\dom J^*$ may happen only in the case when the Radon--Nikodym derivative $\dfrac{d\nu}{d\mu}$ belongs to $\lmu$ and  in that case $\dfrac{d\nu}{d\mu}=J^*1$.
\begin{proposition}\label{P:RNderivative}
    Assume that $\nu$ is absolutely continuous with respect to $\mu$. Then the following assertions are equivalent:
    \begin{enumerate}[\upshape (i)]
      \item The Radon--Nikodym derivative $\dfrac{d\nu}{d\mu}$ belongs to $\lmu$;
      \item $1\in\dom J^*$;
      \item $\D\subseteq\dom J^*$;
      \item There is a nonnegative constant $C\geq0$ such that
      \begin{equation*}
        \Big\lvert\int_{T} \phi~ d\nu\Big\rvert^2\leq C\int_T \abs{\phi}^2~d\mu,\qquad \phi\in\D.
      \end{equation*}
    \end{enumerate}
    In any case, $\dfrac{d\nu}{d\mu}=J^*1$.
\end{proposition}
\begin{proof}
    Assume first that $\dfrac{d\nu}{d\mu}\in\lmu$. Then for any $\phi\in\D$
    \begin{equation*}
        \sipm[\Big]{\phi}{\dfrac{d\nu}{d\mu}}=\int_T \phi~d\nu=\sipn{J\phi}{1}.
    \end{equation*}
    This implies that $1\in \dom J^*$ and that $J^*1=\dfrac{d\nu}{d\mu}.$ That (ii) implies (iii) follows by Lemma \ref{L:measureLemma2}. Assertion (iv) expresses precisely that $1\in\dom J^*$. Finally, if we assume (ii) then $J^*1\geq0$ $\nu$-a.e. by Lemma \ref{L:measureLemma2} and
    \begin{equation*}
        \int_T \phi~d\nu=\sipn{J\phi}{1}=\sipm{\phi}{J^*1}=\int_T \phi\cdot J^*1~d\mu,\qquad \phi\in\D.
    \end{equation*}
    Consequently, $J^*1=\dfrac{d\nu}{d\mu}$ as it is claimed.
\end{proof}
\section{Radon--Nikodym theorems for representable positive functionals on $^*$-algebras}

Throughout this section we fix a (not necessarily unital) $^*$-algebra $\alg$ and two positive functionals $v,w$ on it.  We shall assume $v,w$ to be representable, that is to say, there exists a Hilbert space $\hilv$ (resp., $\hilw$), a $^*$-representation $\pi_v$ (resp., $\pi_w$) of $\alg$ to $\mathscr{B}(\hilv)$ (resp., $\mathscr{B}(\hilw)$) and a cyclic vector $\zeta_v$ (resp., $\zeta_w$) such that
\begin{equation*}
     v(a)=\sipv{\pi_v(a)\zeta_v}{\zeta_v},\qquad w(a)=\sipw{\pi_w(a)\zeta_w}{\zeta_w},\qquad a\in\alg.
\end{equation*}
The cyclicity of $\zeta_v$ (resp., $\zeta_w$) above means that $\pi_v\langle\alg\rangle\zeta_v:=\set{\pi_v(a)\zeta_v}{a\in\alg}$  (resp., $\pi_w\langle\alg\rangle\zeta_w$) is dense in $\hilv$ (resp., in $\hilw$).

Such a triple $(\hilv,\pi_v,\zeta_v)$ is obtained via the well-known GNS construction: set $\vform(a,b)=v(b^*a)$, $\mathscr{N}_{\vform}:=\set{x\in\alg}{\vform(x,x)=0}$, let $\hilv:=\hil_{\vform}$, and define  
\begin{equation*}
\pi_v(a)(b+\mathscr{N}_{\vform}):=ab+\mathscr{N}_{\vform}.
\end{equation*}
The cyclic vector is provided by Riesz representation theorem applied to the densely defined continuous linear functional 
\begin{equation*}
\hilv\to \dupC, \quad a+\mathscr{N}_{\vform}\mapsto v(a).
\end{equation*}
See eg. \cite{Sebestyen84} for the details. The triple $(\hilw,\pi_w,\zeta_w)$ can be constructed analogously.

We recall now the notion of absolute continuity in this context: $w$ is said to be absolutely continuous with respect to $v$ if
\begin{equation*}
    v(a_n^*a_n)\to0\qquad\textrm{and}\qquad w((a_n^*-a^*_m)(a_n-a_m))\to0, \qquad n,m\to\infty,
\end{equation*}
imply $w(a_n^*a_n)\to0$. We notice here that  Gudder \cite{Gudder} called $w$ \emph{strongly $v$-absolutely continuous} in the above case. Observe immediately that $w$ is $v$-absolutely continuous if and only if $\wform$ is $\vform$-absolutely continuous. Equivalently, in the language of Hilbert space operators that means that  the mapping
\begin{equation*}
    J:\hilv\to\hilw,\qquad \pi_v(a)\zeta_v\mapsto\pi_w(a)\zeta_w,\qquad a\in\alg,
\end{equation*}
is a closable operator. 

Our first Radon--Nikodym type result in this setting is a reformulation of Theorem \ref{T:maintheorem}:
\begin{theorem}\label{T:RNrepresentable1}
    If $w$ is $v$-absolutely continuous then there exists $\seq{a}$ of $\alg$ such that
    \begin{equation*}
        w(a)=\limn v(a_n^*a),\qquad a\in\alg.
    \end{equation*}
    Moreover, the convergence is uniform on the set $\set{a\in\alg}{w(a^*a)+v(a^*a)\leq1}$.
\end{theorem}
\begin{proof}
    In the view of Theorem \ref{T:maintheorem}, $\wform$ is $\vform$-pseudo-absolutely continuous. That means that there exists a sequence $\seq{a}$ of $\alg$ such that
    \begin{equation*}
        w(a)=\sipw{\pi_w(a)\zeta_w}{\zeta_w}=\limn \sipv{\pi_v(a)\zeta_v}{\pi_v(a_n)\zeta_v}=\limn v(a_n^*a),
    \end{equation*}
    for all $a\in\alg$, as it is claimed. That the convergence on $\set{a\in\alg}{w(a^*a)+v(a^*a)\leq1}$ is uniform follows immediately form Theorem \ref{T:maintheorem}.
\end{proof}
\begin{corollary}
    Assume that $\alg$ is a Banach $^*$-algebra and let $v,w$ be  representable positive functionals such that $w$ is $v$-absolutely continuous. Then there is a sequence $\seq{a}$ such that $w_n(a):=v(a_n^*a), a\in\alg,$ satisfies $\|w-w_n\|\to0$.
\end{corollary}
\begin{proof}
    According to the proof of Theorem \ref{T:maintheorem}, $\seq{a}$ of Theorem \ref{T:RNrepresentable1} may be chosen such that $\|J^*\xi_n-\pi_v(a_n)\zeta_v\|_{v}\leq1/n$ where $\seq{\xi}$ is a sequence of $\dom J^*\subseteq\hilw$  that satisfies $\|\xi_n-\zeta_w\|_{w}\to0$. Then we have for $a\in\alg, \|a\|\leq1$ that
    \begin{align*}
        \abs{(w-w_n)(a)}\leq&\Abs{\sipw{\pi_w(a)\zeta_w}{\zeta_w}-\sipw{\pi_w(a)\zeta_w}{\xi_n}}\\
        +&\Abs{\sipw{\pi_w(a)\zeta_w}{\xi_n}-\sipv{\pi_v(a)\zeta_v}{\pi_v(a_n)\zeta_v}}\\
\leq& \sup_{a\in\alg,\|a\|\leq1} \|\pi_w(a)\zeta_w\|_{w}\|\xi_n-\zeta_w\|_{w}\\
        +&\sup_{a\in\alg,\|a\|\leq1} \|\pi_v(a)\zeta_v\|_{v}\|J^*\xi_n-\pi_v(a_n)\zeta_v\|_{v}.
    \end{align*}
    Since each $^*$-representation of a Banach $^*$-algebra is continuous, we conclude that  $\|w-w_n\|\to0$.
\end{proof}
The following result is an extension of Gudder's Radon--Nikodym theorem \cite[Theorem 1]{Gudder} for not necessarily unital $^*$-algebras:
\begin{theorem}
    Let $\alg$ be a $^*$-algebra and let $v,w$ be representable functionals on it. The following assertions are equivalent:
    \begin{enumerate}[\upshape (i)]
      \item $w$ is $v$-absolutely continuous;
      \item There exists a positive selfadjoint operator $W$ on $\hilv$ such that $\pi_v\langle\alg\rangle\zeta_v\subseteq\dom W$ and
    \begin{equation}\label{E:W1}
        w(b^*a)=\sipv{W\pi_v(a)\zeta_v}{W\pi_v(b)\zeta_v},\qquad a,b\in\alg.
    \end{equation}
     Furthermore, $\pi_v\langle\alg\rangle\xi\subseteq\dom W$ for $\xi\in\dom W$ and
    \begin{equation}\label{E:W2}
        \sipv{W\pi_v(a)\xi}{W\eta}=\sipv{W\xi}{W\pi_v(a^*)\eta},\qquad a\in\alg, \xi,\eta\in\dom W.
    \end{equation}
    \end{enumerate}
\end{theorem}
\begin{proof}
    An application of Theorem \ref{T:RNforms} shows that (ii) implies (i), and that, by assuming (i),  $W:=(J^*J^{**})^{1/2}$ fulfills \eqref{E:W1}. It remains to show that $W$ fulfills \eqref{E:W2}. Let $a\in\alg$ and $\xi\in\dom W=\dom J^{**}$. Then there exists $\seq{a}$ of $\alg$ such that $\pi_v(a_n)\zeta_v\to\xi$ and that $\pi_w(a_n)\zeta_w\to J^{**}\xi$. Hence $\pi_v(a)\pi_v(a_n)\zeta_v\to\pi_v(a)\xi$ and $J^{**}(\pi_v(a)\pi_v(a_n)\zeta_v)=\pi_w(aa_n)\zeta_w\to\pi_w(a)J^{**}\xi$. This means that $\pi_v(a)\xi\in\dom J^{**}$ and that $J^{**}\pi_v(a)\xi=\pi_w(a)J^{**}\xi$. In particular we infer that
    \begin{equation}\label{E:J**piv}
        J^{**} \pi_v(a)\supseteq\pi_w(a)J^{**},\qquad a\in\alg.
    \end{equation}
    By taking adjoint,
    \begin{equation}\label{E:J*piv}
        \pi_v(a)J^{*}\subseteq J^{*}\pi_w(a),\qquad a\in\alg.
    \end{equation}
    Thus an easy calculation shows that
    \begin{align*}
        \sipv{J^*J^{**}\pi_v(a)\xi}{\eta}=\sipv{\xi}{J^*J^{**}\pi_v(a^*)\eta}
    \end{align*}
    holds for all $\xi,\eta\in\dom J^*J^{**}$. That implies \eqref{E:W2} by noticing that $\dom J^*J^{**}$ is core for $W$.
\end{proof}
\begin{corollary}\label{C:domJ}
    Assume that $w$ is $v$-absolute continuous and that $\zeta_v\in\dom W$ (that satisfies e.g. if $\alg$ is unital, or if $w(a^*a)\leq Cv(a^*a)$,  $a\in\alg$, for some $C\geq0$). Then $J^{**}\zeta_v=\zeta_w$ and
    \begin{equation}\label{E:w(a)=}
        w(a)=\sipv{W\pi_v(a)\zeta_v}{W\zeta_v},\qquad a\in\alg.
    \end{equation}
\end{corollary}
\begin{proof}
    Let $a\in\alg$. By \eqref{E:J**piv} we conclude that
    \begin{align*}
        \sipw{J^{**}\zeta_v}{\pi_w(a)\zeta_w}=\sipw{J^{**}\pi_v(a^*)\zeta_v}{\zeta_w}=\sipw{\pi_w(a^*)\zeta_w}{\zeta_w}=\sipw{\zeta_w}{\pi_w(a)\zeta_w},
    \end{align*}
    hence $J^{**}\zeta_v=\zeta_w$, indeed. Consequently,
    \begin{align*}
        w(a)=\sipw{\pi_w(a)\zeta_w}{\zeta_w}=\sipw{J^{**}\pi_v(a)\zeta_w}{J^{**}\zeta_v}=\sipv{W\pi_v(a)\zeta_v}{W\zeta_v}.
    \end{align*}
    If $\alg$ is unital then $\zeta_v=1$, and if $w\leq Cv$ then $W\in\mathscr{B}(\hilv)$. In any case, $\zeta_v\in\dom W$.
\end{proof}
The next result generalizes \cite[Theorem 1 c]{Gudder} of Gudder:
\begin{proposition}\label{P:domination}
    Assume that $v,w$ are representable functionals on $\alg$ such that $w$ is $v$-absolutely continuous. The following statements are equivalent:
    \begin{enumerate}[\upshape (i)]
      \item $w\leq Cv$ for some $C\geq0$;
      \item $W^2\in Com(\pi_v)$ (that is to say, $W^2\in\mathscr{B}(\hilv)$ and $W^2\pi_v(a)=\pi_v(a)W^2$ for all $a\in\alg$).
    \end{enumerate}
\end{proposition}
\begin{proof}
    Obviously, (i) holds if and only if $W^2=J^{**}J^*$ is continuous. Furthermore, if (i) holds then \eqref{E:J**piv} and \eqref{E:J*piv} become equalities. Hence
    \begin{equation*}
        W^2\pi_v(a)=J^*J^{**}\pi_v(a)=\pi_v(a)J^*J^{**}=\pi_v(a)W^2
    \end{equation*}
    for all $a\in\alg.$
\end{proof}
\begin{corollary}\label{C:irred}
    Assume that $v,w$ are representable functionals on $\alg$ such that $w\leq Cv$ for some $C\geq0.$ If $\pi_v$ is irreducible (that is, $Com(\pi_v)=\dupC.I$) then there exists $\alpha\geq0$ such that $w=\alpha v$.
\end{corollary}
\begin{proof}
    By Proposition \ref{P:domination}, $W^2\in Com(\pi_v)$, hence there exists $\alpha\geq0$ such that $W^2=\alpha I$ according to the irreducibility. Thus, by Corollary \ref{C:domJ},
    \begin{equation*}
        w(a)=\sipv{W\pi_v(a)\zeta_v}{W\zeta_v}=\sipv{\alpha\pi_v(a)\zeta_v}{\zeta_v}=\alpha v(a),
    \end{equation*}
    for all $a\in\alg$.
\end{proof}
We close our paper with characterization of absolute continuity among pure functionals on a $C^*$-algebra, cf. also \cite{Szucs2013}:
\begin{corollary}
    Let $\alg$ be a $C^*$-algebra such that $v,w$ are  positive functionals on $\alg$. If $v$ is pure then the following statements are equivalent:
    \begin{enumerate}[\upshape (i)]
      \item $w$ is $v$-absolutely continuous;
      \item $w=\alpha v$ for some $\alpha\geq0$.
    \end{enumerate}
\end{corollary}
\begin{proof}
    Recall that any positive functional of a $C^*$-algebra is representable. Since $\pi_v\langle\alg\rangle\zeta_v$ is obviously $\pi_v$-invariant, and since $\pi_v$ irreducible, the Kadison transitivity theorem implies that $\pi_v\langle\alg\rangle\zeta_v=\hilv$. This means among others that $\dom W=\hilv$, and hence that $W\in\mathscr{B}(\hilv)$ by the Banach closed graph theorem. Consequently, $w\leq Cv$ with some $C\geq0$. Corollary \ref{C:irred} completes the proof.
\end{proof}
\bibliographystyle{amsplain}

\begin{thebibliography}{10}




\bibitem{DarstGreen}
R. B. Darst and E. Green,
\newblock On a Radon--Nikodym theorem for finitely additive set functions,
\newblock   {\em  Pacific J. Math.,} \textbf{27} (1968), 255--259.

\bibitem{Fefferman}
C. Fefferman,
\newblock A Radon--Nikodym theorem for finitely additive set functions,
\newblock   {\em  Pacific J. Math.,} \textbf{23} (1967), 35--45.

\bibitem{Gudder}
S. P. Gudder,
\newblock A Radon--Nikodym theorem for $^*$-algebras,
\newblock {\em  Pacific J. Math.,} \textbf{80} (1979), 141--149.


\bibitem{Hassi2009}
S.~Hassi, Z.~Sebesty{\'e}n, and H.~de~Snoo,
\newblock Lebesgue type decompositions for nonnegative forms,
\newblock {\em J. Funct. Anal.}, \textbf{257} (2009), 3858--3894.


\bibitem{Sebestyen84}
Z.~Sebesty{\'e}n,
\newblock On representability of linear functionals on $^*$-algebras,
\newblock {\em  Period. Math. Hungar.}, \textbf{15} (1984), 233--239.

\bibitem{sebestytarcsaytitkos}
Z.~Sebesty{\'e}n, Zs. Tarcsay, and T. Titkos,
\newblock Lebesgue decomposition theorems,
\newblock {\em Acta Sci. Math. (Szeged)}, \textbf{79} (2013), 219--233.

\bibitem{sebestytitkosRN}
Z.~Sebesty{\'e}n and T. Titkos,
\newblock A Radon--Nikodym type theorem for forms,
\newblock {\em Positivity}, \textbf{17} (2013), 863--873.

\bibitem{Simon}
B.~Simon,
\newblock A canonical decomposition for quadratic forms with applications to monotone convergence theorems,
\newblock {\em J. Funct. Anal.}, \textbf{28} (1978) 377-–385.

\bibitem{Szucs2013}
\newblock{Zs.~Sz\H{u}cs},
\newblock{The Lebesgue decomposition of representable forms over algebras},
{\em J. Operator Theory}, \textbf{70} (2013), 3-31.


\bibitem{titkos}
\newblock{T.~Titkos},
\newblock{Lebesgue decomposition of contents via nonnegative forms},
{\em Acta Math. Hungar.}, \textbf{140} (2013), 151--161.

\end{thebibliography}

\end{document}